\documentclass[12 pt]{article}

\usepackage{amsmath}
\usepackage{amssymb}
\usepackage[english]{babel}
\usepackage{amsthm}
\usepackage[latin1]{inputenc}
\usepackage{hyperref}
\usepackage{graphicx}
\usepackage{subfigure}
\usepackage{makeidx}

\newtheorem{teo}{Theorem}
\newtheorem{lemma}{Lemma}
\newtheorem{prop}{Proposition}
\newtheorem{defin}{Definition}
\newtheorem{cor}{Corollary}

\newenvironment{sistema}%
{\left\lbrace\begin{array}{@{}l@{}}}%
{\end{array}\right.}

{\left( \begin{array}{@{}l@{}}}%
{\end{array}\right)}

\begin{document}

\title{\textbf{Relativistic pendulum and invariant curves}}

%

\author{\textbf{Stefano Marò} \\
\textit{\small{Dipartimento di Matematica - Università di Torino}}\\
\textit{\small{Via Carlo Alberto 10, 10123 Torino - Italy}}\\
\textit{\small{e-mail:stefano.maro@unito.it }}
}
\date{}

\maketitle

\begin{abstract}
We apply KAM theory to the equation of the forced relativistic pendulum to prove that all the solutions have bounded momentum. Subsequently, we detect the existence of quasiperiodic solutions in a generalized sense. This is achieved using a modified version of the Aubry-Mather theory for compositions of twist maps.
\end{abstract}

%
%
%

\section{Introduction}
In this paper we are concerned with some aspects of the dynamics of the differential equation
\begin{equation}\label{introeq}
\frac{d}{dt} \Bigl(  \frac{\dot x}{\sqrt{1-{\dot x}^2}}   \Bigr) + a\sin x = f(t),
\end{equation}
where $a>0$ is a parameter and $f:\mathbb{R}\rightarrow\mathbb{R}$ is a continuous and $T$-periodic real function satisfying
\begin{equation}\label{nm}
\int_0^Tf(t)dt=0.
\end{equation}
This equation, sometimes called the forced relativistic pendulum, has been considered by several authors. In \cite{torres1} Torres proved the existence of a $T$ periodic solution after imposing some restrictions on the period and the size of $f$. Later, Brezis and Mawhin \cite{brezismawhin} proved the existence of a $T$-periodic solution for any $f$. The existence of a second $T$-periodic solution has been proved in \cite{bereanutorres, bereanujebeleanmawhin}. See also \cite{marotopol, fondatoader} for an alternative approach to the periodic problem. The equation (\ref{introeq}) can be seen as a relativistic counterpart of the classical Newtonian pendulum 
\begin{equation}\label{introcl}
\ddot{x}+a\sin x=f(t).
\end{equation}
This equation has been analysed from many points of view. In particular Levi \cite{levikam} and You \cite{you} proved that all the solutions of (\ref{introcl}) have bounded velocity $\dot{x}(t)$ whenever (\ref{nm}) holds. The relativistic framework implies that $|\dot{x}(t)|<1$ and so the boundedness of the velocity is automatic. However we will prove that the results by Levi and You have a relativistic parallel when the velocity is replaced by the momentum
$$
p(t)=\frac{\dot{x}(t)}{\sqrt{1-{\dot{x}(t)}^2}}.
$$  
The main result of this paper says that if $f(t)$ satisfies (\ref{nm}) then all solutions of (\ref{introeq}) satisfy
$$
\sup_{t\in\mathbb{R}}|p(t)|<\infty.
$$
Moreover we will prove that condition (\ref{nm}) is essential for this conclusion. In addition we will prove the existence of generalized quasi-periodic solutions with two frequencies
$$
\omega_1=\frac{2\pi}{T},\quad \omega_2\in(-1,1)
$$
Note that we find  solutions for each frequency $\omega_2$ and these solutions are quasi-periodic when the phase space of the pendulum is a cylinder. These solutions become subharmonic solutions when $\omega_1$ and $\omega_2$ are commensurable.

To prove these results we consider the Hamiltonian formulation of (\ref{introeq}) where the position $q=x$ and the momentum $p=\frac{\dot{x}}{\sqrt{1-\dot{x}^2}}$ are conjugate variables. After some changes of coordinates we will write the associated Poincaré map in a form such that Moser Twist Theorem is applicable and so invariant curves exist. This property already implies the boundedness of the momentum. To apply Moser's theorem, estimates is some $C^k$ norm are needed. These estimates usually are tedious and cumbersome and one has to find the right way case by case. This is why trying to repeat the direct computation by Levi or the change of variable by You, one is lead to non trivial technical difficulties. Anyway, a more general technique, inspired by \cite{ortegalisboa} and based on the differentiability with respect to the parameter of the solution of a differential equation, will simplify significantly the computations. Moreover, it could also provide a simpler proof of the results of Levi and You for the Newtonian case. We also stress the fact that we can consider the case in which the period of the forcing and the period of the potential are not the same.  Furthermore, the generality of this argument allows to consider a general nonlinearity $g(x)$ in (\ref{introeq}).\\
To prove the existence of periodic and generalized quasi-periodic solutions, one can use the theory of Aubry and Mather \cite{mathertop}. In principle, to apply this theory we need to known that the Poincaré map of equation (\ref{introeq}) has twist. In the paper \cite{marotopol} it was shown that it does not hold unless a restriction on the parameters is imposed, namely the condition $a\leq\frac{\pi^2}{T^2}$ is necessary. Since we want to obtain results for arbitrary parameters we will apply a less standard version of Aubry-Mather theory. In \cite{matherams} it is shown that the main conclusion of this theory still holds when the map is obtained as a finite composition of twist maps. The Poincaré map $\Pi$ of equation (\ref{introeq}) can be seen as a finite composition $\Pi=f_1\circ\dots\circ f_N$ where every $f_i$ is a "small-time-map" that is twist without any restriction. To consider finite composition of twist map is a great novelty in \cite{matherams} and many results on twist maps admit an extension to this setting. In \cite{matherams} the twist of each map $f_i$ goes to infinity as the action goes to infinity. The relativistic effects prevents the velocity from being too large and this makes impossible to satisfy this assumption of large twist. For this reason we must modify Mather's theorem in order to adjust it to our situation. With this modified theorem we can produce periodic and quasi-periodic solutions whose oscillating properties are determined by the rotation number of the corresponding Mather set.

\section{Motions with bounded momentum}
Consider the equation
\begin{equation}\label{introeq2}
\frac{d}{dt} \Bigl(  \frac{\dot x}{\sqrt{1-{\dot x}^2}}   \Bigr) - g(x) = f(t)
\end{equation}
and assume that the functions $f$ and $g$ satisfy the following conditions 
\begin{itemize}
\item[(A1)] $g\in C^7(\mathbb{R})$, $g(x+S)=g(x)$, $\int^S_0g(x)dx=0$ 
\item[(A2)] $f\in C(\mathbb{R})$, $f(t+T)=f(t)$, $\int^T_0 f(t)dt=0$.
\end{itemize}
where $T$ and $S$ are two positive numbers. Notice that when $g(x)=-a\sin x$ and $S=2\pi$ we recover equation (\ref{introeq}). \\
Equation (\ref{introeq2}) is in the Lagrangian framework. Actually it can be expressed in the form
$$
\frac{d}{dt}\left(\frac{\partial L}{\partial\dot{x}}\right)-\frac{\partial L}{\partial x}=0 
$$ 
where
$$
L(x,\dot{x},t)=-\sqrt{ 1- \dot{x}^2 }+G(x)+f(t)x.
$$ 
Here $G$ represents a primitive of $g$. Notice that $G$ is $S$-periodic and of class $C^8$. 
\\
To our purposes, it will be convenient to pass to the Hamiltonian formulation,   
\begin{equation}\label{rpend1}
\begin{sistema}
\dot{q}=H_p=\frac{p}{\sqrt{1+p^2}}\\
\dot{p}=-H_q=g(q)+f(t)
\end{sistema}
\end{equation}
with $H(t,q,p)=\sqrt{1+p^2}-G(q)-f(t)q$.
We arrive to this system after having performed the classical Legendre transformation
\begin{equation}\label{cv}
\begin{sistema}
q=x\\
p=\frac{\dot{x}}{\sqrt{1-\dot{x}^2}}.
\end{sistema}
\end{equation}
From now on the symplectic coordinate $p$ will be called the momentum. The Hamiltonian vector field $(H_p,-H_q)$ is bounded so all solutions of (\ref{rpend1}) are globally defined and the same holds for the solutions of (\ref{introeq2}) undoing the change of variables.\\
Notice that, due to the relativistic structure, the velocity of any solution is bounded and satisfies
$$
|\dot{x}(t)|<1\quad\mbox{for each }t\in\mathbb{R}.
$$
We will prove that also the momentum is bounded. This is equivalent to the more restrictive condition on the velocity,
$$
\sup_{t\in\mathbb{R}}|\dot{x}(t)|<1.
$$
Precisely
\begin{teo}\label{teo1}
Assume that (A1) and (A2) hold. Then every solution $(q(t),p(t))$ of (\ref{rpend1}) satisfies
$$
\sup_{t\in\mathbb{R}}|p(t)|<\infty.
$$
\end{teo}
Moreover, we will show that the null mean value of the function $f$ is an essential condition in the above theorem.
\begin{prop}\label{prop1}
Assume that (A1) holds and that $f$ is a continuous and $T$-periodic function satisfying 
$$
\bar{f}=\frac{1}{T}\int^T_0f(t)dt\neq 0.
$$
Then there exists $R>0$ such that if $(q(t),p(t))$ is a solution of (\ref{rpend1}) with $|p(0)|\geq R$, the momentum satisfies
$$
\lim_{t\to\infty}|p(t)|=\infty.
$$ 
\end{prop} 
Proof of these results will be presented in the following sections. Moreover, we will perform the proof for the case $S=1$, being conjugated to the general one through a change of scale.

\section{The approximated Poincaré map}
The solution of (\ref{rpend1}) satisfying the initial condition
$$
q(0)=q_0,\quad p(0)=p_0
$$
will be denoted by $(q(t;q_0,p_0),p(t;q_0,p_0))$. The main tool of our work will be the Poincaré map, the area preserving diffeomorphisms of the plane defined by
$$
\Pi:\mathbb{R}^2\rightarrow\mathbb{R}^2, \quad \Pi(q_0,p_0)=(q(T;q_0,p_0),p(T;q_0,p_0)).
$$
The periodicity of $g$ (remember that we suppose $S=1$) implies that $\Pi$ satisfies
$$
\Pi(q_0+1,p_0)=\Pi(q_0,p_0)+(1,0)
$$  
and so $\Pi$ induces a diffeomorphism of the cylinder $\mathbb{T}\times\mathbb{R}$, where $\mathbb{T}=\mathbb{R}/\mathbb{Z}$.\\
On the other hand, the periodicity of $f$ allows to describe the dynamics of system (\ref{rpend1}) in terms of the map $\Pi$. In particular, the condition
$$
\sup_{t\in\mathbb{R}}|p(t;q_0,p_0)|<\infty
$$
is equivalent to 
$$
\sup_{n\in\mathbb{Z}}|p_n|<\infty
$$
where $(q_n,p_n)=\Pi^n(q_0,p_0)$. Similarly,
$$
\lim_{t\to\infty}|p(t;q_0,p_0)|=\infty
$$
is equivalent to
$$
\lim_{n\to\infty}|p_n|=\infty.
$$
Notice that the boundedness of the vector field $(H_p,-H_q)$ plays a role in the proof of this equivalence.\\
In view of this equivalences, to prove theorem \ref{teo1} we shall look for non contractible invariant curves for the map $\Pi$. Two disjoint invariant curves define an annulus that is invariant under the diffeomorphism $\Pi$, so we can say that they act as barriers. Our aim will be to apply Moser's small twist theorem to the Poincaré map $\Pi$. With the promise of being more precise later on, we recall that Moser's theorem gives the existence of invariant curves for a class of maps of the cylinder whose lift has the form
\begin{equation}\label{mos}
\begin{sistema}
\theta_1=\theta+\omega+\delta[\alpha(r) +  R_1(\theta,r)]\\
r_1=r+\delta R_2(\theta,r)
\end{sistema}
\end{equation}   
supposing that the reminders $R_1$ and $R_2$ were small in some $C^k$ norm. Here $\delta$ plays the role of a small parameter.\\
The coordinates $(q,p)$ are not the best ones to have the Poincaré map written in form (\ref{mos}), so perform the following symplectic change of variables
\begin{equation*}
\begin{sistema}
q=Q \\
p=P+G(q)+F(t)
\end{sistema}
\end{equation*}
where $F(t)$ is a primitive of $f$. Note that that $F(t)$ is $T$-periodic and $C^1$. We get the system
\begin{equation}\label{rpend3}
\begin{sistema}
\dot{Q}=\frac{ P+G(Q)+F(t)}{\sqrt{1+(P+G(Q)+F(t))^2}}\\
\dot{P}=g(Q)(1-\frac{ P+G(Q)+F(t)}{\sqrt{1+(P+G(Q)+F(t))^2}})
\end{sistema}
\end{equation} 
Now we can introduce the small parameter $\delta>0$ through the following change of scale
\begin{equation}\label{change}
Q=u,\quad P=\frac{1}{\delta v} \quad v\in[1/2,7/2].
\end{equation} 
It is important to note that the strip $\mathbb{R}\times [1/2,7/2]$ corresponds in the original variables to the time dependent region
$$
A_\delta=\{(q,p)\in\mathbb{R}^2 :\: \frac{2}{7\delta}+G(q)+F(t)\leq p\leq \frac{2}{\delta}+G(q)+F(t) \}
$$
and so from the boundedness of $F$ and $G$
\begin{equation}\label{dove}
 p\to\infty \mbox{ as } \delta\to 0 \mbox{ uniformly in }v. 
 \end{equation}
System (\ref{rpend3}) transforms into 
\begin{equation}\label{rpend2}
\begin{sistema}
\dot{u}=\frac{1+\delta v[G(u)+F(t)]}{\sqrt{\delta^2 v^2+(1+\delta v[G(u)+F(t)])^2}}\\
\dot{v}=-\delta v^2g(u)[1-\frac{1+\delta v[G(u)+F(t)]}{\sqrt{\delta^2 v^2+(1+\delta v[G(u)+F(t)])^2}}].
\end{sistema}
\end{equation}
The change of variables (\ref{change}) is not symplectic, but the Poincaré map of systems (\ref{rpend2}) is still conjugated to $\Pi$.\\
Note that if $\delta=0$ system (\ref{rpend2}) transforms into  
\begin{equation*}
\begin{sistema}\label{sisze}
\dot{u}=1\\
\dot{v}=0
\end{sistema}
\end{equation*} 
and taking any initial condition $(u_0,v_0)\in\mathbb{R}\times (1/2,7/2)$ we have that the solution is well-defined for $t\in [0,T]$. So, by continuous dependence, there exists $\Delta>0$ such that the if $\delta\in [0,\Delta]$ the solution is still well-defined for $t\in [0,T]$.
The coordinates $(u,v)$ are the good ones to have the Poincaré map written in form (\ref{mos}). To have a rough idea of why this is true, one can see trough a formal computation that system (\ref{rpend2}) has the following expansion for small $\delta$
\begin{equation*}
\begin{sistema}
\dot{u}=1-\frac{1}{2}\delta^2v^2+O(\delta^3)\\
\dot{v}=O(\delta^3).
\end{sistema}
\end{equation*}  
Notice the fundamental fact that up to second order $F$ and $G$ do not play any role. Now one can obtain the Poincaré map integrating and evaluating at $t=T$.\\
We are going to make this argument rigorous and the key is the theory of differentiability with respect to the parameters. So, inspired by \cite{ortegalisboa}, let us recall some general facts. Consider a differential equation depending on a parameter
\begin{equation}\label{sisgen}
\frac{dz}{dt}=\Psi(t,z,\delta)
\end{equation}
where $\Psi:[0,T]\times\mathcal{D}\times [0,\Delta]\rightarrow\mathbb{R}^n$ is of class $C^{0,\nu+2,\nu+2}$, $\nu\geq 1$ and $\mathcal{D}$ is an open connected subset of $\mathbb{R}^n$ and $\Delta>0$. The general theory of differential equations says that the solution $z(t,z_0,\delta)$ is of class $C^{0,\nu+2,\nu+2}$ in its three arguments. The following lemma will be crucial for our purpose, and generalizes the result \cite[Proposition 6.4]{ortegalisboa}.
\begin{lemma}\label{derpar}
Let $K$ be a compact set of $\mathcal{D}$ such that for every $z_0\in K$ and $\delta \in [0,\Delta]$ the solution is well defined in $[0,T]$. Then, for every $(t,z,\delta)\in [0,T]\times K\times [0,\Delta]$ the following expansion holds
$$
z(t,z_0,\delta)=z(t,z_0,0)+ \delta\frac{\partial z}{\partial\delta}(t,z_0,0)+\frac{\delta^2}{2}\frac{\partial^2 z}{\partial\delta^2}(t,z_0,0)+\frac{\delta^2}{2} R(t,z_0,\delta)
$$
where
$$
||R(t,\cdot ,\delta)||_{C^\nu(K)}\to 0 \quad\mbox{as }\delta \to 0
$$
uniformly in $t\in [0,T]$.
\end{lemma}
\begin{proof}
For a function $\phi\in C^{0,\nu+2,\nu+2}([0,T]\times K\times [0,\Delta])$, the Taylor formula with remainder in integral form gives
\begin{equation*}
\phi(t,z_0,\delta)=\phi(t,z_0,0)+ \frac{\partial \phi}{\partial\delta}(t,z_0,0)\delta+\frac{\delta^2}{2}\frac{\partial^2 \phi}{\partial\delta^2}(t,z_0,0)+ R_2(t,z_0,\delta) 
\end{equation*}
where
$$
R_2(t,z_0,\delta)=\frac{1}{2}\int_0^\delta \frac{\partial^3 \phi}{\partial\delta^3}(t,z_0,\xi)(\delta-\xi)^2d\xi. 
$$ 
Integrating by parts one gets
$$
R_2(t,z_0,\delta)=\frac{1}{2} \{2\int_0^\delta \frac{\partial^2 \phi}{\partial\delta^2}(t,z_0,\xi)(\delta-\xi)d\xi-\frac{\partial^2 \phi}{\partial\delta^2}(t,z_0,0)\delta^2         \}
$$
and through the change of variable $\xi=\delta s$ we get
$$
R_2(t,z_0,\delta)=\delta^2\int_0^1 (1-s)[\frac{\partial^2 \phi}{\partial\delta^2}(t,z_0,\delta s)-\frac{\partial^2 \phi}{\partial\delta^2}(t,z_0,0)]ds.
$$
from which it is easy to conclude using the regularity of the solution.
\end{proof}
Note that, by means of this lemma we have a semi-explicit formula for the solution of (\ref{sisgen}). This is very useful to compute its Poincaré map. So, let us apply the previous lemma to system (\ref{rpend2}). 
First of all, calling $Z=(u,v)$, system (\ref{rpend2}) can be written in the form 
$$
\dot{Z}=\Psi(t;Z,\delta).
$$
The initial condition will be denoted by $Z(0)=z_0=(u_0,v_0)$ and the corresponding solution by $z(t;z_0,\delta)=(u(t;u_0,v_0,\delta),v(t;u_0,v_0,\delta))$. We will suppose, by periodicity, that $z_0\in[0,1]\times [1,3]$. From (\ref{sisze}) we have that
\begin{equation}\label{zerder}
z(t;u_0,v_0,0)=(u_0+t,v_0).
\end{equation}
To compute the first derivative with respect to the parameter let us call $X(t;z_0,\delta)=\frac{\partial z}{\partial\delta}(t;z_0,\delta)$. 
We need $X(t;z_0,0)$ that solves the Cauchy problem
\begin{equation*}
\begin{sistema}
\dot{X}=A(t)X+a(t)\\
X(0)=0.
\end{sistema}
\end{equation*}
where
$$
A(t)=\frac{\partial \Psi}{\partial Z}(t;z(t;z_0,0),0),\quad a(t)=\frac{\partial \Psi}{\partial\delta}(t;z(t;z_0,0),0).
$$
A simple computation gives
\begin{equation}\label{jacob}
\frac{\partial \Psi}{\partial Z}(t;Z,0)=0  \quad  \frac{\partial\Psi}{\partial\delta}(t;Z,0)=0 
\end{equation}
so that
\begin{equation}\label{primder}
X(t;u_0,v_0,0)=0.
\end{equation}
Now let us compute the second derivative. Let us call $Y(t;z_0,\delta)=\frac{\partial^2 z}{\partial\delta^2}(t;z_0,\delta)$ with components $(\xi(t;z_0,\delta),\eta(t;z_0,\delta))$. We need $Y(t;z_0,0)$ that solves the Cauchy problem
\begin{equation*}
\begin{sistema}
\dot{Y}=A(t)Y+b(t) \\
Y(0)=0
\end{sistema}
\end{equation*}
where
\begin{multline*}
b(t)=\frac{\partial^2 \Psi}{\partial\delta^2}(t;z(t;z_0,0),0)+2\frac{\partial^2 \Psi}{\partial Z\partial\delta}(t;z(t;z_0,0),0)X(t;z_0,0)\\ +\frac{\partial^2 \Psi}{\partial Z^2}(t;z(t;z_0,0),0)[X(t;z_0,0),X(t;z_0,0)]
\end{multline*}
and $\frac{\partial^2 \Psi}{\partial Z^2}(t;z(t;z_0,0),0)$ is interpreted as a bilinear form from $\mathbb{R}^2\times\mathbb{R}^2$ into $\mathbb{R}^2$.  
A simple computation gives
$$
\frac{\partial^2 \Psi}{\partial\delta^2}(t;z(t;z_0,0),0)=(-v_0^2,0).
$$  
From (\ref{jacob}) and (\ref{primder}) we get the system
\begin{equation*}
\begin{sistema}
\dot{\xi}=-v_0^2,\quad \xi(0)=0\\
\dot{\eta}=0,\quad \eta(0)=0
\end{sistema}
\end{equation*}
leading to
\begin{equation}\label{seconder}
Y(t;u_0,v_0,0)=(-v_0^2t,0).
\end{equation}
Next we apply lemma \ref{derpar} using (\ref{zerder}), (\ref{primder}) and (\ref{seconder}). We have that
$$
Z(t;u_0,v_0,\delta)=(u_0+t,v_0)+\frac{\delta^2}{2}(-v_0^2t,0)+\frac{\delta^2}{2}R(t;u_0,v_0,\delta),
$$
where the remainder $R$ satisfies the estimate
$$
||R(t,\cdot ,\delta)||_{C^5([0,1]\times[1,3])}\to 0 \quad\mbox{as }\delta \to 0
$$
uniformly in $t\in [0,T]$. Finally, evaluating at $t=T$ we get the following expression for the Poincaré map
\begin{equation}\label{papprox}
\begin{sistema}
u_1=u_0+T-\frac{\delta^2}{2} Tv_0^2 + \frac{\delta^2}{2} R_1(u_0,v_0,\delta)\\
v_1=v_0+\frac{\delta^2}{2} R_2(u_0,v_0,\delta)
\end{sistema}
\end{equation}
and
\begin{equation}\label{resto}
||R_1(\cdot ,\cdot,\delta)||_{C^5(\mathbb{R}/\mathbb{Z}\times[1,3])}+||R_2(\cdot ,\cdot,\delta)||_{C^5(\mathbb{R}/\mathbb{Z}\times[1,3])}\to 0\quad\mbox{as }\delta\to 0.
\end{equation}

\section{Invariant curves vs. Lyapunov functions}
We saw that theorem \ref{teo1} will be proved as soon as we could place any initial condition $(q_0,p_0)$ between two invariant curves. In view of (\ref{dove}) it is sufficient to prove the existence of invariant curves for the map (\ref{papprox}) as $\delta\to 0$. More precisely, we are going to prove the existence of a sequence of invariant curves $\Gamma_n$ approaching uniformly the top of the cylinder. Analogously one can prove the existence of a sequence of invariant curves approaching the bottom of the cylinder. Finally we will prove proposition \ref{prop1} to show that the null mean value of the forcing $f$ is essential to have invariant curves.\\
Concerning the boundedness, as anticipated, we have found the variables $(u,v)$ in order to have the Poincaré map written in form (\ref{papprox}) and apply Moser's small twist theorem whose original version is in \cite{moserinvcurve}. There are many versions of this theorem and we shall employ one coming from the works of Herman \cite{herman1,herman2} and explicitly stated in \cite{ortegaadv}. To recall it, let $\mathbb{T}=\mathbb{R}/\mathbb{Z}$ and consider the infinite cylinder $\mathcal{C}=\mathbb{T}\times\mathbb{R}$ and its strip $\mathcal{A}=\mathbb{T}\times [a,b]$ with $b-a\geq\frac{3}{2}$. The theorem deals with maps $g:\mathcal{A}\rightarrow \mathcal{C}$ with lift 
\begin{equation*}
\begin{sistema}
\theta_1=\theta+\omega+\delta[\alpha(r)+R_1(\theta,r)]  \\
r_1=r+\delta R_2(\theta,r)
\end{sistema}
\end{equation*} 
where $\alpha\in C^4[a,b]$, and $R_1,R_2 \in C^4(\mathcal{A})$. The number $\omega\in\mathbb{R}$ is arbitrary and $\delta\in (0,1]$ is a parameter. Suppose that the function $\alpha$ satisfies 
$$
c^{-1}_0\leq \alpha^\prime(r)\leq c_0 \quad \forall r\in[a,b],\quad ||\alpha||_{C^4[a,b]}\leq c_0
$$
for some constant $c_0>1$.
Moreover, we suppose that $g$ satisfies the intersection property, in the sense that 
$$
g(\Gamma)\cap\Gamma\neq\emptyset
$$   
for every non-contractible Jordan curve $\Gamma\subset \mathcal{A}$.
\begin{teo}[\cite{ortegaadv}]\label{moser}
Let $g:\mathcal{A}\rightarrow \mathcal{C}$ be a mapping in the previous conditions. Then there exists $\epsilon>0$, depending on only on $c_0$, such that if
$$
||F||_{C^4(\mathcal{A})}+||G||_{C^4(\mathcal{A})}\leq\epsilon
$$ 
the map $g$ has an invariant curve.
\end{teo}

Now everything is ready for the proof of theorem \ref{teo1}. Excepting for the intersection property, it is easy to see that the Poincaré map expressed in the form (\ref{papprox})-(\ref{resto}) satisfies all the hypothesis of theorem \ref{moser}. In this case, $\theta=u_0$, $r=v_0$, $\omega=T$, $\alpha(v_0)=-\frac{T}{2}v_0^2$ and $\delta$ is small enough. Concerning the intersection property, notice that from a result in \cite{marotopol}, the null mean value of $f$ implies that the Poincaré map associated to system (\ref{rpend1}) is exact symplectic in the sense that the differential form 
$$
p_1dq_1-pdq
$$
is exact in the cylinder. Moreover, it is known that an exact symplectic map has the intersection property. Finally we can say that also map (\ref{papprox}) has the intersection property because this property is preserved by conjugacy. So, an application of theorem \ref{moser} proves theorem \ref{teo1}.  
\smallskip
We have just proved that hypothesis (A1) and (A2) imply that the momentum is bounded. To complete the study of the boundedness we need to prove proposition \ref{prop1}. We will perform the proof supposing that 
$$
\bar{f}=\frac{1}{T}\int_0^Tf(s)ds> 0,
$$ 
the other case being similar.  
We just need to prove that there exists $R$ sufficiently large such that if $|p_0|\geq R$ then the corresponding orbit of the Poincaré map $\Pi$ is unbounded. In this case, a less subtle expansion of $\Pi$, coming directly from system (\ref{rpend1}), will be sufficient. So, integrate (\ref{rpend1}) and get, for $t\in [0,T]$
\begin{equation}\label{hu}
\begin{sistema}
q(t;q_0,p_0)=q_0+t+\tilde{\varepsilon}(t,q_0,p_0) \\  

p(t;q_0,p_0)=p_0+ \int_0^t g(q(s;q_0,p_0))ds + \int_0^t f(s)ds  
\end{sistema}
\end{equation}
where
$$
\tilde{\varepsilon}(t,q_0,p_0)=\int_0^t\left\{ \frac{p(s;q_0,p_0)}{\sqrt{1+p^2(s;q_0,p_0)}}-1 \right\}ds
$$
As $p(t;q_0,p_0)\to\infty$ as $p_0\to\infty$ uniformly in $q_0$ and $t\in [0,T]$, we have $\tilde{\varepsilon}\to 0$ as $p_0\to\infty$, uniformly in $q_0$ and $t\in [0,T]$.\\
Adding and subtracting $\int_0^t g(q_0+s)ds=G(q_0+t)-G(q_0)$ in the second equation of (\ref{hu}) we get
$$
p(t;q_0,p_0)=p_0+ G(q_0+t) - G(q_0) + \int_0^tf(s)ds + \varepsilon(t,q_0,p_0)
$$
where
$$
\varepsilon(t,q_0,p_0)=\int_0^t\{g(q_0+s+\tilde{\varepsilon}(s,q_0,p_0))-g(q_0+s)  \}ds.
$$
The mean value theorem implies that $\varepsilon\to 0$ as $p_0\to\infty$ uniformly in $q_0$ and $t\in [0,T]$. Evaluating in $t=T$ we get the following expansion of $\Pi$: 
\begin{equation*}
\begin{sistema}
q_1=q_0+T+\tilde{\varepsilon}(T,q_0,p_0) \\  

p_1=p_0+ G(q_0+T) - G(q_0) + T\bar{f} + \varepsilon(T,q_0,p_0)  
\end{sistema}
\end{equation*}
where $\varepsilon$ and $\tilde{\varepsilon}$ tends to zero uniformly in $q_0$ as $p_0$ tends to $+\infty$.

Now, inspired by \cite{alonsoortega}, consider the function
$$
V(q,p)=q-G(p).
$$
and notice that
$$
V( \Pi(q,p))=V(q,p)+\Gamma(q,p)
$$
where
$$
\Gamma(q,p)=-G(q+T+\tilde{\varepsilon}(T,q,p))+G(q+T) + \varepsilon(T,q,p) + T\bar{f}.
$$
Now, using the fact that $G$ is bounded, one can find $V_*$ such that if $V(q_0,p_0)\geq V_*$ then $p_0$ is sufficiently large in order to have $\Gamma(q_0,p_0)>\frac{T\bar{f}}{2}$. For such a $p_0$ we have

$$
V( \Pi(q_0,p_0))>V(q_0,p_0)+ \frac{T\bar{f}}{2}>V^*.
$$
So, by induction we can prove that
$$
V(q_n,p_n)>V(q_0,p_0)+ n\frac{T\bar{f}}{2},\quad n\geq 1.
$$
Finally we have that 
$$
\lim_{n\to\infty}V(q_n,p_n)=+\infty
$$
and remembering the definition of $V$ and the boundedness of $G$ we get that $p_n\to +\infty$. 

\section{Generalized quasi-periodic and periodic solutions}
We have just proved that all the solutions of (\ref{rpend1}) have bounded momentum and a natural question is to describe the kind of recurrent motions that can be expected. Periodic solutions of different types always exist (\cite{marotopol},\cite{fondatoader}), and now we are going to look for quasi-periodic solutions. Precisely, we will prove 
\begin{teo}\label{teo3}
For every $\omega\in (-T,T)$, there exists a family of solutions of (\ref{rpend1}), $X_\xi(t)=(q_\xi(t),p_\xi(t))$, with $\xi\in\mathbb{R}$ such that  
\begin{equation}\label{piu}
X_{\xi+1}(t)=X_\xi(t)+(1,0)\quad\mbox{and}\quad X_\xi(t+T)=X_{\xi+\omega}(t).
\end{equation}
Moreover, the initial conditions
$$
\xi\mapsto q_\xi(0)\quad\mbox{and}\quad\xi\mapsto p_\xi(0)
$$
are of bounded variation and 
$$
\lim_{t\to\infty}\frac{q_\xi(t)}{t}=\frac{\omega}{T}.
$$ 
\end{teo}
To understand why these solutions satisfy a kind of weak quasi-periodicity we define, inspired by \cite{ortegaasy},
$$
\Phi_\xi(\theta_1,\theta_2)=X_{\theta_2-\frac{\omega}{T}\theta_1+\xi}(\theta_1).
$$
It satisfies 
\begin{equation*}
\Phi_\xi(\theta_1+T,\theta_2)=\Phi_\xi(\theta_1,\theta_2),\quad \Phi_\xi(\theta_1,\theta_2+1)=\Phi_\xi(\theta_1,\theta_2)+(1,0) 
\end{equation*}
and this says that the function $\Phi_\xi$ is doubly periodic once it takes values on the phase space $\mathbb{T}\times\mathbb{R}$. The solution is recovered by the formula
$$ 
X_\xi(t)=\Phi_\xi(t,\frac{\omega}{T} t)
$$
when $\Phi_\xi$ is continuous as a function of the three variables $(\xi,\theta_1,\theta_2)$. This function is quasi-periodic. Again we are assuming that it takes values on $\mathbb{T}\times\mathbb{R}$. In the discontinuous case the solution will not be quasi-periodic in the classical sense but the bounded variation of the initial conditions implies that quasi-periodicity in the sense of Mather will appear. See \cite{mathertop} for more details.
When the number $\omega$ is rational, say $\omega=\frac{a}{b}$ with $a$ and $b$ relatively prime, then 
$$
X_\xi(t+bT)=X_\xi(t)+(a,0)
$$
and the solution is periodic with period $bT$. Once more we are assuming that $X_\xi$ takes values on $\mathbb{T}\times\mathbb{R}$. Classically these solutions are called subharmonic solutions of the second kind.\\

To prove theorem \ref{teo3}, consider the change of variable
 \begin{equation*}
 \begin{sistema}
 Q= q \\
 P= p-F(t) 
 \end{sistema}
 \end{equation*}
 where $\dot{F}=f$. System (\ref{rpend1}) transforms into
 \begin{equation}\label{rpend4}
 \begin{sistema}
 \dot{Q}=\frac{P+F(t)}{\sqrt{1+(P+F(t))^2}} \\
 \dot{P}=g(Q) 
 \end{sistema}
 \end{equation}  
The Poincaré map of system (\ref{rpend4}) has a particular form. Consider a partition of the interval $[0,T]$ in $N$ sub intervals of equal length 
\begin{equation}\label{condtwist}
L=\frac{T}{N}<\frac{\pi}{\sqrt{||g^\prime||_\infty}}
\end{equation}
and consider the map $\Pi_{L,\tau}(Q_0,P_0)=(Q(\tau+L;\tau,Q_0,P_0),P(\tau+L;\tau,Q_0,P_0))=(Q_1,P_1)$ where $(Q(t;\tau,Q_0,P_0),P(t;\tau,Q_0,P_0))$ is the solution of (\ref{rpend4}) with initial condition $(Q_0,P_0)$ at time $\tau$. The Poincaré map $\Pi$ of the system can be written as composition of such maps, precisely we have that 
$$
\Pi=\Pi_{T,0}=\Pi_{L,(N-1)L}\circ\dots\circ\Pi_{L,L}\circ\Pi_{L,0}.
$$ 
So let us study such maps. It is worth recalling some definition inspired by \cite{matherams}. Consider a $C^2$ diffeomorphism $f(\theta,r)=(\Theta(\theta,r),R(\theta,r))=(\theta_1,r_1)$ of the infinite cylinder $\mathbb{T}\times \mathbb{R}$ that is isotopic to the identity. Passing to the lift, the components satisfy
$$
\Theta(\theta+1,r)=\Theta(\theta,r)+1,\quad\mbox\quad R(\theta+1,r)=R(\theta,r).
$$
We stress the fact that in his work Mather required only a $C^1$ diffeomorphism, but for our purposes we will need more smoothness. The diffeomorphism is said 
\begin{itemize}
\item exact symplectic if the differential form $Rd\Theta-rd\theta$ is exact in $\mathbb{T}\times\mathbb{R}$,
\item twist if $\partial\Theta/\partial r>0$, while, if there exists $\beta>0$ such that $\partial\Theta/\partial r>\beta$ we will say that $f$ is $\beta$-twist,
\item to preserve the ends of the infinite cylinder, if $R(\theta,r)\to\pm\infty$ as $r\to\pm\infty$ uniformly in $\theta$,  
\item to twist each end infinitely, if $\Theta(\theta,r)-\theta\to\pm\infty$ as $r\to\pm\infty$ uniformly in $\theta$.
\end{itemize}
Now we can recall the
\begin{defin}
Let $\mathcal{P}^\infty=\bigcup_{\beta>0}\mathcal{P}_{\beta}$, where $\mathcal{P}_\beta$ is the class of $C^2$ diffeomorphisms of the infinite cylinder that
\begin{enumerate}
\item are isotopic to the identity 
\item are exact symplectic 
\item are $\beta$-twist
\item preserve the ends of the infinite cylinder, 
\item twist each end infinitely.
\end{enumerate} 
\end{defin}
For our purposes we will need
\begin{defin}
Let $\mathcal{P}^{\rho_+,\rho,_-}$ be the class of $C^2$ diffeomorphisms  of the infinite cylinder that satisfy properties $1.$, $2.$, $4.$ of the previous definition and
\begin{itemize}
\item[3'.] are twist 
\item[5'.] are such that $\Theta(\theta,r)-\theta\to\rho_\pm$ as $r\to\pm\infty$ uniformly in $\theta$,
\item[6.]  there exists $M$ such that $|R(\theta,r)-r|\leq M$ for every $(\theta,r)\in\mathbb{T}\times\mathbb{R}$
\end{itemize}
\end{defin} 
Now we can start the study of the map $\Pi_{L,\tau}$. Notice that by the periodicity of (\ref{rpend4}) it can be seen as a map defined on the cylinder $\mathbb{T}\times\mathbb{R}$. Moreover we have that
\begin{lemma}\label{exact}
For every $\tau\in [0,T]$, the map $\Pi_{L,\tau}$ is exact symplectic in $\mathbb{T}\times\mathbb{R}$
\end{lemma}
\begin{proof}
Inspired by \cite{ortegakunze} consider the function
\begin{equation*}
V_\tau( Q_0, P_0)=\int_\tau^{\tau+L} \left\{ -\frac{F^2(t)+P(t;\tau, Q_0, P_0)F(t)+1}{\sqrt{1+(P(t;\tau Q_0, P_0)+F(t))^2}}+G(Q(t;\tau Q_0, P_0))\right\} dt.
\end{equation*}

First of all, it follows from the periodicity of (\ref{rpend4}) that $Q(t;\tau, Q_0+1,  P_0)=Q(t;\tau, Q_0,  P_0)+1$ and $P(t;\tau, Q_0+1, P_0)=P(t;\tau, Q_0, P_0)$. Hence we have
\begin{equation*}
V_\tau( Q_0+1, P_0)=V_\tau( Q_0, P_0)
\end{equation*}
Now let us compute the differential $dV_\tau$. 

We have
\begin{equation}\label{htf}
\begin{split}
\frac{\partial V_{\tau}}{\partial Q_0}= & \int_\tau^{\tau+L}\left\{\frac{P}{[1+(P+F)^2]^{3/2}} \frac{\partial P}{\partial Q_0}+g(Q)\frac{\partial Q}{\partial Q_0}\right\} dt \\
   &=\int_\tau^{\tau+L}\left\{\frac{P}{[1+(P+F)^2]^{3/2}} \frac{\partial P}{\partial Q_0}+\dot{P}\frac{\partial Q}{\partial Q_0}\right\} dt
\end{split}
\end{equation}
using the second equation in (\ref{rpend4}).
Now, integrating by parts and using the first equation in (\ref{rpend4}) we get
\begin{equation*}
\begin{split}
\int_\tau^{\tau+L}\dot{P} \frac{\partial Q}{\partial Q_0}dt  &= [P\frac{\partial Q}{\partial Q_0}]_\tau^{\tau+L}-\int_\tau^{\tau+L} P \frac{\partial \dot{Q}}{\partial Q_0} dt \\ 
 &= [P\frac{\partial Q}{\partial Q_0}]_\tau^{\tau+L}-\int_\tau^{\tau+L} \frac{P}{[1+(P+F)^2]^{3/2}} \frac{\partial P}{\partial Q_0}
\end{split}
\end{equation*}
that, substituting in (\ref{htf}) gives
$$
\frac{\partial V_{\tau}}{\partial Q_0}=P(\tau+L)\frac{\partial Q}{\partial P_0}(\tau+L)- P(\tau)\frac{\partial Q}{\partial Q_0}(\tau).
$$
Analogously we can get 
$$
\frac{\partial V_{\tau}}{\partial P_0}=P(\tau+L)\frac{\partial Q}{\partial P_0}(\tau+L)- P(\tau)\frac{\partial Q}{\partial P_0}(\tau).
$$  
Hence $dV=P_1dQ_1-P_0dQ_0$ and the lemma is proved.
\end{proof}  
This is not the only property satisfied by the map. In fact we have
\begin{prop}\label{pll}
For every $\tau\in [0,T]$, we have $\Pi_{\tau,L}\in\mathcal{P}^{-L,L}$
\end{prop}
\begin{proof}
First of all, from lemma \ref{exact} we have that the map $\Pi_{L,\tau}$ is exact symplectic and by a similar argument as in \cite{marotopol} condition (\ref{condtwist}) implies that for every $\tau\in [0,T]$, the map $\Pi_{L,\tau}$ is twist and isotopic to the identity. From equation (\ref{rpend4}) we have
 \begin{equation*}
 \begin{sistema}
 Q(t;\tau, Q_0,P_0)=Q_0+\int_{\tau}^{t}\frac{P(s;\tau,Q_0,P_0)+F(s)}{\sqrt{1+(P(s;\tau,Q_0,P_0)+F(s))^2}}ds \\
 P(t;\tau, Q_0,P_0)=P_0+\int_{\tau}^{t}g(Q(s;\tau,Q_0,P_0))ds. 
 \end{sistema}
 \end{equation*}
 Evaluating the second equation in $t=\tau+L$, the boundedness of $g$ gives that $\Pi_{\tau,L}$ preserves the end of the infinite cylinder. Moreover, evaluating the first equation in $t=\tau+L$ and using the second we easily get   
$$
\lim_{P_0\to\pm\infty} (Q_1-Q_0)=\pm L
$$
uniformly in $Q_0$. Finally, property $6.$ is a trivial consequence of the boundedness of $g$.
\end{proof}
So, summing up we have that the Poincaré map of system (\ref{rpend4}) can be written as a composition of maps in $\mathcal{P}^{-L,L}$ and this justifies the study that we are going to develop in the next section.

\section{Composition of twist maps and proof of Theorem \ref{teo3}}

Consider a finite family $\{ f_i\}_{i=1,\dots,N}$ such that $f_i\in\mathcal{P}^\infty$ for every $i$. Let $F=f_1\circ\dots \circ f_N$. We have that $F$ is a $C^2$ exact symplectic diffeomorphism of $\mathbb{T}\times\mathbb{R}$ that preserves the ends and such that twists the ends infinitely. However, it has not to be twist.\\
In \cite{matherams}, Mather proved that one can associate to $F$ a continuous function $h(\theta,\theta_1)$, called variational principle, that acts as a generating function for a twist diffeormorphism. The variational principle satisfies, among others, these relevant properties:
\begin{itemize}
\item[(H1)] $h(\theta+1,\theta_1+1)=h(\theta,\theta_1)$,
\item[(H5)] There exists a positive continuous function $\rho$ on $\mathbb{R}^2$ such that
$$
h(\gamma,\theta_1)+h(\theta,\gamma_1)-h(\theta,\theta_1)-h(\gamma,\gamma_1)\geq\int^\gamma_\theta\int^{\gamma_1}_{\theta_1}\rho
$$
if $\theta<\gamma$ and $\theta_1<\gamma_1$,  
\item[(H6$\alpha$)] there exists $\alpha>0$ such that 
\begin{equation*}
\begin{split}
\theta & \to\alpha\theta^2/2-h(\theta,\theta_1) \mbox{ is convex for every } \theta_1  \\
\theta_1 &\to\alpha\theta_1^2/2-h(\theta,\theta_1) \mbox{ is convex for every } \theta.
\end{split}
\end{equation*}
\end{itemize}
The function $h$ in general is not differentiable but from (H6) one can prove that the one side partial derivatives $\partial_1h(\theta\pm,\theta_1)$ and $\partial_2h(\theta,\theta_1\pm)$ exist. Mather proved that there exist particular configurations $(\bar{\theta}_i)$ that minimize an action. They are called minimal configurations and are such that the partial derivatives $\partial_1h(\bar{\theta}_i,\bar{\theta}_{i+1})$ and $\partial_2h(\bar{\theta}_{i-1},\bar{\theta}_i)$ both exist and satisfy
\begin{equation}\label{stat}
\partial_1h(\bar{\theta}_i,\bar{\theta}_{i+1})+\partial_2h(\bar{\theta}_{i-1},\bar{\theta}_i)=0\quad \mbox{for every } i. 
\end{equation}
This property allows to construct a complete orbit $(\bar{\theta}_i,\bar{r}_i)$ of $F$ defining 
$$
\bar{r}_i=-\partial_1h(\bar{\theta}_i,\bar{\theta}_{i+1})=\partial_2h(\bar{\theta}_{i-1},\bar{\theta}_i). 
$$
Once we have a minimal configuration $(\bar{\theta}_i)$, we can define for $(p,q)\in\mathbb{Z}\times\mathbb{Z}$ its translate $T_{p,q}\bar{\theta}$ by $(T_{p,q}\bar{\theta})_i=\bar{\theta}_{i+q}-p$. In an analogous way we can define the translate of an orbit. The translate of a minimal configuration is minimal. Moreover, given two configurations $\Theta=(\theta_i)$ and $\Gamma=(\gamma_i)$ we say that $\Theta<\Gamma$ if $\theta_i<\gamma_i$ for every $i$. Two configurations $\Theta$ and $\Gamma$ are comparable if either $\Theta=\Gamma$ or $\Theta>\Gamma$ or $\Theta<\Gamma$. 
Using these characterizations Mather proved
\begin{teo}[\cite{matherams}]\label{matteo}
Let $F=f_1\circ\dots \circ f_N$ with $f_i\in\mathcal{P}^\infty$ for $i=1,\dots,N$. Then for every $\omega\in\mathbb{R}$ there exists an orbit $(\bar{\theta}_i,\bar{r}_i)$ of $F$ such that any two translates of $(\bar{\theta}_i)$ are comparable and the sequence $(\bar{\theta}_i)$ is increasing. Moreover,  
$$
\lim_{i\to\infty}\frac{\bar{\theta}_i}{i}=\omega
$$ 
and $\omega$ is called rotation number. 
\end{teo}
The connection between theorem \ref{matteo} and the result in the first paper by Mather \cite{mathertop} is stated in the following 
\begin{cor}\label{relmat}
From the orbit $(\bar{\theta}_i,\bar{r}_i)$ in the previous theorem, we can construct two functions $\phi ,\eta:\mathbb{R}\rightarrow\mathbb{R}$ satisfying, for every $t\in\mathbb{R}$
\begin{equation*}
\phi(t+1)=\phi(t)+1,\quad \eta(t+1)=\eta(t)
\end{equation*}
\begin{equation}\label{roto}
F(\phi(t),\eta(t))=(\phi(t+\omega),\eta(t+\omega))
\end{equation}
where $\phi$ is monotone (strictly if $\omega\notin\mathbb{Q}$) and $\eta$ is of bounded variation.
\end{cor}

\begin{proof}
Inspired by \cite{moserbook}, let us consider, for every $\omega$, the set
\begin{equation}\label{sig}
 \Sigma=\{t\in\mathbb{R}:\: t=j\omega-k \mbox{ for some } (j,k)\in\mathbb{Z}^2\}.
 \end{equation}
  We have to distinguish whether $\omega$ is rational or not.\\   
  $-$ If $\omega$ is irrational, $\Sigma$ is a dense additive subgroup of $\mathbb{R}$ and every pair $(j,k)$ gives rise to a different number. We proceed by steps.\\
  \textit{STEP 1: definition of $\phi$ on $\Sigma$}. 
  If $t\in\Sigma$ we define
 \begin{equation}\label{deffi}
 \phi(t)=\bar{\theta}_j-k.
 \end{equation}  
We claim that the function $\phi:\Sigma\rightarrow\mathbb{R}$ is strictly increasing: we have to prove that 
$$
j\omega-k<j^\prime\omega-k^\prime\Rightarrow\bar{\theta}_j-k<\bar{\theta}_{j^\prime}-k^\prime
$$
that is, calling $r=j^\prime-j$ and $s=k^\prime-k$,
$$
0<r\omega-s\Rightarrow\bar{\theta}_j<\bar{\theta}_{j+r}-s.
$$
The case $r=0$ is obvious, so suppose $r\neq 0$. Suppose by contradiction that for some $j\in\mathbb{Z}$  
\begin{equation}\label{ind}
\bar{\theta}_j\geq\bar{\theta}_{j+r}-s
\end{equation} 
we have, from the comparison property of the translated, that either 
$$
\bar{\theta}_i>\bar{\theta}_{i+r}-s\quad \mbox{ for every }i.
$$
or
$$
\bar{\theta}_i=\bar{\theta}_{i+r}-s\quad \mbox{ for every }i.
$$
In the second case the orbit would be periodic and this is not compatible with an irrational rotation number.
So, from (\ref{ind}) we can prove by induction that for every $n\in\mathbb{N}$
$$
\bar{\theta}_j>\bar{\theta}_{j+nr}-ns.
$$
Now suppose that $r>0$. Taking the limit for $n\to\infty$ after having divided by $nr$ we get
$$
0\geq\omega-\frac{s}{r}.
$$
that leads to a contradiction as we multiply by $r$. Notice that we can repeat the same argument and get the same contradiction for $r<0$.\\
Moreover, $\phi$ satisfies the periodicity property
$$
\phi(t+1)=\phi(t)+1 \quad \mbox{for each } t\in\Sigma.
$$
\textit{STEP 2: extension of $\phi$ outside $\Sigma$}.
Given $\tau\in\mathbb{R}-\Sigma$, the limits
$$
\phi(\tau\pm)=\lim_{t\to\tau^\pm, t\in\Sigma}\phi(t)
$$
exist and $\phi(\tau -)\leq\phi(\tau +)$. To extend $\phi$ to a monotone function on the whole real line it is sufficient to impose $\phi(\tau)\in[\phi(\tau -),\phi(\tau +)]$ and we choose $\phi(\tau)=\phi(\tau -)$. In this way $\phi:\mathbb{R}\rightarrow\mathbb{R}$ is strictly increasing and satisfies
$$
\phi(t+1)=\phi(t)+1 \quad \mbox{for each } t\in\mathbb{R}.
$$
\textit{STEP 3: definition of $\eta$ on $\Sigma$}. 
Define, for $t\in\Sigma$
\begin{equation}\label{et}
\eta(t)=\partial_2 h(\phi(t-\omega),\phi(t))
\end{equation}
where $h$ is the variational principle associated to $F$. We claim that for every $t,s\in \Sigma$
\begin{equation}\label{lips}
|\eta(s)-\eta(t)|\leq\alpha(\phi(s)-\phi(t))
\end{equation}
where $\alpha$ comes from (H6$\alpha$). Supposing $t<s$ we have from the monotonicity
$$
\phi(t-\omega)<\phi(s-\omega),\quad \phi(t)<\phi(s),\quad \phi(t+\omega)<\phi(s+\omega).
$$
Inspired by \cite[Proposition 2.6]{matherams}, we notice that if in (H5) we set $\gamma=\phi(s-\omega)$, $\theta=\phi(t-\omega)$, $\theta_1=\phi(t)-\epsilon$, $\gamma_1=\phi(t)$ with $\epsilon>0$, divide by $\epsilon$ and let $\epsilon\to 0$ we get
$$
\partial_2h(\phi(s-\omega),\phi(t)-)\leq\partial_2h(\phi(t-\omega),\phi(t)).
$$ 
remembering that the partial derivatives exist on the orbits. Moreover, from (H6$\alpha$) and remembering that the one side partial derivatives of a convex function exist and are non decreasing, we have
$$
\partial_2h(\phi(s-\omega),\phi(s))\leq\partial_2h(\phi(s-\omega),\phi(t)-)+\alpha(\phi(s)-\phi(t)).
$$
Combining these two inequalities we have
$$
\eta(s)\leq \eta(t)+\alpha(\phi(s)-\phi(t)).
$$
Using using (\ref{stat}) we can see that also $\eta(t)=-\partial_1 h(\phi(t),\phi(t+\omega))$ so we can get analogously
$$
\eta(t)\leq \eta(s)+\alpha(\phi(s)-\phi(t))
$$
and conclude.\\
\textit{STEP 4: extension of $\eta$ outside $\Sigma$}. 
If $\tau\notin\Sigma$ we define
\begin{equation}\label{defeta}
\eta(\tau)=\lim_{t\uparrow\tau,t\in\Sigma}\eta(t)
\end{equation}
This is a correct definition. Indeed, from (\ref{lips}) we have that
$$
|\eta(t_{n+k})-\eta(t_n)|\leq\alpha|\phi(t_{n+k})-\phi(t_n)|,
$$ 
and, being $\phi(t_n)$ a Cauchy sequence, we have that $\eta(t_n)$ converges and the limit (\ref{defeta}) exists. In principle the limit could depend on the sequence. This is not the case, indeed in case that $\eta(t^1_n)\to l_1$ and $\eta(t^2_n)\to l_2$ we can construct a new increasing sequence $(\tau_n)$ having $(t^1_n)$ and $(t^2_n)$ as sub-sequences. So also $\eta(\tau_n)$ has to converge to a limit that is the same as $l_1$ and $l_2$. So the definition (\ref{defeta}) makes sense. With this definition we have that estimate (\ref{lips}) holds for every $t,s\in\mathbb{R}$. Since $\phi$ is monotone and hence of bounded variation, we have that $\eta$ is of bounded variation.\\
Now, from the periodicity property of $h$ and $\phi$ we get that $\eta(t+1)=\eta(t)$.\\ 
\textit{STEP 5: property (\ref{roto}) holds}.
Let us assume first that $t\in\Sigma$. Then $t=j\omega-k$ and
$$
\phi(t)=\bar{\theta}_j-k, \quad \phi(t+\omega)=\bar{\theta}_{j+1}-k.
$$
Moreover,
$$
\eta(t)=\partial_2 h(\phi(t-\omega),\phi(t))=\partial_2 h(\bar{\theta}_{j-1},\bar{\theta}_j)=\bar{r}_j
$$
and similarly $\eta(t+\omega)=\bar{r}_{j+1}$. Since $(\bar{\theta}_j-k,\bar{r}_j)$ is an orbit of $F$ we conclude that
$$
F(\phi(t),\eta(t))=(\phi(t+\omega),\eta(t+\omega)).
$$  
Let us assume now that $t\in\mathbb{R}\setminus\Sigma$. So we select a sequence $(t_n)$ converging to $t$ with $t_n\in\Sigma$ and $t_n<t$. Then we can pass to the limit in the identity
$$
F(\phi(t_n),\eta(t_n))=(\phi(t_n+\omega),\eta(t_n+\omega)).
$$
The irrational case is done.\\
$-$ The case $\omega=\frac{p}{q}$ rational is simpler. We can suppose that $p$ and $q$ are relative primes and that the corresponding sequence $(\bar{\theta}_i)$ is periodic (in the sense that $\bar{\theta}_{i+q}=\bar{\theta}_i+p$). First of all notice that in this case, the subgroup $\Sigma$ defined in (\ref{sig}) is discrete, precisely, 
$$
\Sigma=\{\frac{d}{q}: \: d\in\mathbb{Z}\}.
$$
The representation $t=j\omega-k$ is not unique, indeed $t=j\frac{p}{q}-k=j^\prime\frac{p}{q}-k^\prime$ whenever $k^\prime-k=Np$ and $j^\prime-j=Nq$ for some $N\in\mathbb{N}$. Anyway the periodicity of $(\bar{\theta}_i)$ implies that
$$
j\frac{p}{q}-k=j^\prime\frac{p}{q}-k^\prime\Rightarrow \bar{\theta}_j-k=\bar{\theta}_{j^\prime}-k^\prime.
$$
So we can define $\phi$ on $\Sigma$ as in (\ref{deffi}). As before one can prove that $\phi:\Sigma\rightarrow\mathbb{R}$ is increasing (non strictly). We extend it to a monotone function on the whole $\mathbb{R}$ as a piecewise constant function that is continuous from the left and taking only the values $\bar{\theta}_j-k$.  
 
Finally, as before, one can prove that $\phi(t+1)=\phi(t)+1$. Moreover, the fact that $\phi$ takes only values at points of a minimal orbit, we can define directly for $t\in\mathbb{R}$ 
\begin{equation*}
\eta(t)=\partial_2 h(\phi(t-\omega),\phi(t)).
\end{equation*}
This function is of bounded variation and condition (\ref{roto}) is satisfied as well. To prove this we just have to repeat the same arguments as in the irrational case. Note that this time it is not necessary to pass to the limit.   
\end{proof}

In our case, theorem \ref{matteo} cannot be applied, as the hypothesis of the infinite twist at infinity is not satisfied. So we will present a modified version of the theorem. First we give the following notation: let $\Gamma_k$ be a sequence of non-contractible Jordan curves that are invariant under a map $f$. This curves are called invariant curves. We say that $\Gamma_k \uparrow +\infty$ uniformly if there exists a sequence $r_k\to +\infty$ as $k\to +\infty$ such that $\Gamma_k\subset\mathbb{T}\times (r_k,+\infty)$. The reader can easily guess the meaning of $\Gamma_k \downarrow -\infty$ uniformly. 


We can prove
\begin{teo}\label{comp}
Consider a finite family $\{ f_i\}_{i=1,\dots,N}$ where $f_i\in\mathcal{P}^{\rho_+,\rho_-}$. Let $F=f_1\circ\dots\circ f_N$. Suppose that $F$ possesses a sequence $(\Gamma_k)$ of invariant curves such that $\Gamma_k \uparrow +\infty$ uniformly as $k\to +\infty$ and $\Gamma_k \downarrow -\infty$ uniformly as $k\to -\infty$   
 Then, for every $\omega\in (N\rho_-,N\rho_+)$ there exist two functions $\phi ,\eta:\mathbb{R}\rightarrow\mathbb{R}$ satisfying the same properties as in Corollary \ref{relmat}.   
\end{teo}

The proof of this theorem relies on the following lemmas

\begin{lemma}\label{modif}
Consider $f\in\mathcal{P}^{\rho_+,\rho,_-}$. Fix an interval $[a,b]$. Then, there exists $\tilde{f}\in\mathcal{P}^\infty$ such that $f=\tilde{f}$ on $\mathbb{T}\times [a,b]$.
\end{lemma}   
\begin{proof}
It is convenient to work with the generating function $h(\theta,\theta_1)$. Remember that it is a $C^3$ function defined on the set $\tilde{\Sigma}=\{ \rho_-<\theta_1-\theta<\rho_+ \}\subset\mathbb{R}^2$ such that $h(\theta+1,\theta_1+1)=h(\theta,\theta_1)$ and satisfies the Legendre condition $\partial_{12}h<0$. It generates $f$ in the sense that the map $f$ is defined implicitly by the equations
\begin{equation}\label{gert}
\begin{sistema}
\partial_1 h(\theta,\theta_1)= -r  \\
\partial_2 h(\theta,\theta_1)= r_1.
\end{sistema}
\end{equation}
More details can be found in \cite{ortegakunze}. 
Notice that the strip $\mathbb{T}\times[a,b]$ of the cylinder corresponds to the set $\tilde{\Sigma}_2=\{ \alpha(\theta)\leq\theta_1-\theta\leq\beta(\theta) \}\subset\tilde{\Sigma}$ where $\alpha$ and $\beta$ are implicitly defined by
\begin{equation*}
\begin{split}
- & \partial_1 h(\theta,\theta+\alpha(\theta))=a \\
- & \partial_1 h(\theta,\theta+\beta(\theta))=b.
\end{split}
\end{equation*}
The functions $\alpha$ and $\beta$ are $C^2$, $1$-periodic and the Legendre condition implies that $\alpha(\theta)<\beta(\theta)$. Moreover, we have that $\alpha(\theta)\downarrow\rho_-$ as $a\to -\infty$ and $\beta(\theta)\uparrow\rho_+$ as $b\to +\infty$. Now take two larger strips $\tilde{\Sigma}_1=\{ \tilde{a}\leq\theta_1-\theta\leq\tilde{b} \}$ and $\tilde{\Sigma}_\epsilon=\{ \tilde{a}+\epsilon<\theta_1-\theta<\tilde{b}-\epsilon \}$ such that $\tilde{\Sigma}_2\subset\tilde{\Sigma}_\epsilon\subset\tilde{\Sigma}_1\subset\tilde{\Sigma}$ (cfr figure).
\begin{figure}[h]
\begin{center}
\includegraphics[scale=0.7]{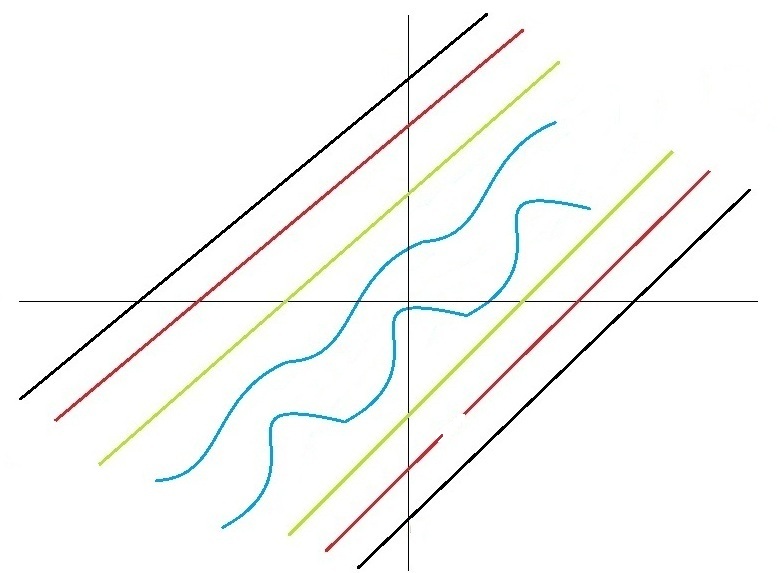}
\put(-195,20){$\rho_-$}
\put(-195,46){$\tilde{a}$}
\put(-195,75){$\tilde{a}+\epsilon$}
\put(-195,190){$\tilde{b}-\epsilon$}
\put(-195,224){$\tilde{b}$}
\put(-195,254){$\rho_+$}
\put(-200.6,288){$\wedge$}
\put(-190,286){$\theta_1$}
\put(-190,165){$\beta(\theta)$}
\put(-137,185){$\alpha(\theta)$}
\put(-20,140){$>$} 
\put(-20,127){$\theta$}
\put(-155,293){$\tilde{\Sigma}$}
\put(-135,285){$\tilde{\Sigma}_1$}
\put(-115,270){$\tilde{\Sigma}_\epsilon$}
\put(-120,235){$\tilde{\Sigma}_2$}
\end{center}
\end{figure}
Notice that, by compactness, there exists $\delta>0$ such that $\partial_{12}h<-\delta$ on $\tilde{\Sigma}_1$. Now, fix $\varepsilon>0$ small and extend $\partial_{12}h$ out of $\{\rho_-+\varepsilon\leq\theta_1-\theta_0\leq\rho_+-\varepsilon\}$ as a $C^1$ bounded function (it is not important how you do it). So we can suppose that there exists a constant $M_1>0$ such that
\begin{equation}\label{hbound}
\sup_{(\theta_0,\theta_1)\in\mathbb{R}^2}|\partial_{12}h|\leq M_1.
\end{equation}
 Consider $\chi$ a $C^\infty$ cut-off function of $\mathbb{R}^2$ such that
\begin{equation*}
\begin{sistema}
\chi=1 \mbox{ on }\tilde{\Sigma}_\epsilon \\
\chi=0 \mbox{ on } \{ \theta_1-\theta>\tilde{b}\}.
\end{sistema}
\end{equation*}
Moreover we can suppose that $\chi=\chi(\theta_1-\theta)$, $0\leq\chi\leq 1$ and $\chi>0$ on $\{\tilde{b}-\epsilon<\theta_1-\theta<\tilde{b}\}$.   
Define the new function 
$$
\Delta=\chi\partial_{12}h+(\chi-1)\delta.
$$
We notice that $\Delta\in C^1$, $\Delta(\theta_1+1,\theta+1)=\Delta(\theta_1,\theta)$ and 
\begin{equation*}
\begin{sistema}
\Delta=\partial_{12}h \mbox{ on }\tilde{\Sigma}_\epsilon \\
\Delta=-\delta \mbox{ on } \{ \theta_1-\theta>\tilde{b}\}
\end{sistema}
\end{equation*}
With a similar argument as in \cite{marobounce} we can consider the following Cauchy problem for the wave equation
\begin{equation*}
\begin{sistema}
\partial_{12}u=\Delta(\theta,\theta_1)\\
u(\theta,\theta+\tilde{a})=h(\theta,\theta+\tilde{a})\\
(\partial_2u-\partial_1u)(\theta,\theta+\tilde{a})=(\partial_2h-\partial_1h)(\theta,\theta+\tilde{a}).
\end{sistema}
\end{equation*}
 The solution $h^+$ is defined on the set $\{\theta_1-\theta>\tilde{a}+\epsilon \}$, is such that $h^+\in C^2$, $h^+(\theta_1+1,\theta+1)=h^+(\theta_1,\theta)$, $\partial_{12}h^+=\Delta$ and $h^+=h$ on $\tilde{\Sigma}_\epsilon$. Now perform an analogous argument to modify $\partial_{12}h$ also in the zone $\{ \theta_1-\theta<\tilde{a} \}$ finding $h^-$. Finally glue $h^+$ and $h^-$ through the common part $\tilde{\Sigma}_\epsilon$ to get a function $\tilde{h}$. Notice that $\partial_{12}\tilde{h}\leq -\delta$ on $\mathbb{R}^2$. The function $\tilde{h}$ generates via (\ref{gert}) a diffeomorphism $\tilde{f}(\theta,r)=(\theta_1,r_1)$ such that the relation 
$$
\frac{\partial \theta_1}{\partial r}   = -\frac{1}{\partial_{12}\tilde{h}}
$$
holds. So the diffeomorphism $\tilde{f}$ is $\beta$-twist with $\beta=1/\max\{-\partial_{12}\tilde{h}\}$ and satisfies property 5'. Moreover, as $h=\tilde{h}$ on $\tilde{\Sigma}_\epsilon$, the diffeomorphism $\tilde{f}$ coincides with $f$ on $\mathbb{T}\times [a,b]$. 
\end{proof}

It is not hard to guess that we are going to use this lemma to modify the diffeomorphism $F$ through its components $f_i$. So, it is worth introducing some notation. Given $f\in\mathcal{P}^{\rho_-,\rho_+}$ and an interval $[a,b]$ then the modified diffeomorphism $\tilde{f}$ with support $[a,b]$ is the diffeomorphism coming from lemma \ref{modif}. Given $F=f_1\circ\dots\circ f_N$ with $f_i\in\mathcal{P}^{\rho_-,\rho_+}$, we will call $\tilde{F}$ with support $[a,b]$ the diffeomorphism given by $\tilde{F}=\tilde{f}_1\circ\dots\circ \tilde{f}_N$ where every $\tilde{f}_i$ is supported in $[a,b]$. Moreover, notice that, if $f_i\in\mathcal{P}^\infty$ then trivially $\tilde{f}_i\equiv f_i$. Finally, $F$ has coordinates $(\Theta(\theta,r),R(\theta,r))$ while $f_i$ has coordinates $(\Theta^{(i)}(\theta,r),R^{(i)}(\theta,r))$ and the corresponding modifications have coordinates $(\tilde{\Theta}(\theta,r),\tilde{R}(\theta,r))$ and $(\tilde{\Theta}^{(i)}(\theta,r),\tilde{R}^{(i)}(\theta,r))$. 

\begin{lemma}\label{estimerre}
Consider $f\in\mathcal{P}^{\rho_-,\rho_+}$. There exists $K>0$ such that for every modified $\tilde{f}$ with support $[a,b]$
$$
|\tilde{R}(\theta,r)-r|\leq K\quad\mbox{for every } (\theta,r)\in\mathbb{T}\times\mathbb{R}
$$
uniformly in $[a,b]$.   
\end{lemma}
\begin{proof}
We have to prove that, given a modification with support $[a,b]$, we have the estimate with the constant $K$ independent on $[a,b]$. Consider the generating function $\tilde{h}$ of $\tilde{f}$. We have to estimate the quantity
$$
|\partial_2\tilde{h}(\theta,\theta_1)+\partial_1\tilde{h}(\theta,\theta_1)|.
$$
Notice that, with the notation of the previous lemma, in $[\tilde{b}-\epsilon,\tilde{a}+\epsilon]$ we have $h\equiv\tilde{h}$ so the estimate comes directly from property $6.$ in the definition of the class $f\in\mathcal{P}^{\rho_-,\rho_+}$. If $\theta_1-\theta>\tilde{b}$ or $\theta_1-\theta<\tilde{a}$ then $\tilde{R}(\theta,r)=r$ and $K=0$. So we only have to study the cases $\tilde{b}-\epsilon\leq\theta_1-\theta\leq\tilde{b}$ and $\tilde{a}\leq\theta_1-\theta\leq\tilde{a}+\epsilon$. Let us study the first, being the second similar. We need d'Alambert formula, valid for a function $V\in C^2(\mathbb{R}^2)$: 
\begin{equation*}
\begin{split}
V(\theta,\theta_1)=&-\int^{\theta_1}_{\theta+\delta}d\eta \int^{\eta-\delta}_{\theta}\partial_{12}V(\xi,\eta)d\xi + V(\theta,\theta+\delta)+ \\ &  \int^{\theta_1}_{\theta+\delta}\partial_{2}V(\eta-\delta,\eta)d\eta
\end{split}
\end{equation*}
where $\delta\in\mathbb{R}$. Applying it to $\tilde{h}$ and choosing $\delta=\tilde{b}-\epsilon$ we get
\begin{equation*}
\begin{split}
\tilde{h}(\theta,\theta_1)=&-\int^{\theta_1}_{\theta+\tilde{b}-\epsilon}d\eta \int^{\eta-\tilde{b}+\epsilon}_{\theta}\Delta(\xi,\eta)d\xi + h(\theta,\theta+\tilde{b}-\epsilon)+ \\ &  \int^{\theta_1}_{\theta+\tilde{b}-\epsilon}\partial_{2}h(\eta-\tilde{b}+\epsilon,\eta)d\eta
\end{split}
\end{equation*}
Let us compute the partial derivatives.
The fundamental theorem of calculus gives
$$
\partial_1\tilde{h}(\theta,\theta_1)=\int^{\theta_1}_{\theta+\tilde{b}-\epsilon}\Delta(\theta,\eta)d\eta+\partial_1h(\theta,\theta+\tilde{b}-\epsilon).
$$
Remembering the definition of $\Delta$ we have, integrating by parts
\begin{equation*}
\begin{split}
&\int^{\theta_1}_{\theta+\tilde{b}-\epsilon}\Delta(\theta,\eta)d\eta=\int^{\theta_1}_{\theta+\tilde{b}-\epsilon}\chi(\eta-\theta)\partial_{12}h(\theta,\eta)d\eta+\delta\int^{\theta_1}_{\theta+\tilde{b}-\epsilon} \{\chi(\eta-\theta)-1\}d\eta=\\
&\chi(\theta_1-\theta)\partial_1h(\theta,\theta_1)-\partial_1h(\theta,\theta+\tilde{b}-\epsilon)-\int^{\theta_1}_{\theta+\tilde{b}-\epsilon}\chi^\prime(\eta-\theta)\partial_{1}h(\theta,\eta)d\eta\\
&+\delta\int^{\theta_1}_{\theta+\tilde{b}-\epsilon} \{\chi(\eta-\theta)-1\}d\eta
\end{split}
\end{equation*}
where we used the fact that $\chi(\tilde{b}-\epsilon)=1$. So
\begin{equation*}
\begin{split}
\partial_1\tilde{h}(\theta,\theta_1)= & \chi(\theta_1-\theta)\partial_1h(\theta,\theta_1) + \delta\int^{\theta_1}_{\theta+\tilde{b}-\epsilon} \{\chi(\eta-\theta)-1\}d\eta\\
& -\int^{\theta_1}_{\theta+\tilde{b}-\epsilon}\chi^\prime(\eta-\theta)\partial_{1}h(\theta,\eta)d\eta.
\end{split}
\end{equation*}
Similarly,
\begin{equation*}
\begin{split}
\partial_2\tilde{h}(\theta,\theta_1)= & \chi(\theta_1-\theta)\partial_2h(\theta,\theta_1) - \delta\int^{\theta_1-\tilde{b}+\epsilon}_{\theta} \{\chi(\theta_1-\xi)-1\}d\xi\\
& -\int^{\theta_1-\tilde{b}+\epsilon}_{\theta}\chi^\prime(\theta_1-\xi)\partial_{2}h(\xi,\theta_1)d\xi.
\end{split}
\end{equation*}
Now we can concentrate on the quantity
$$
|\partial_2\tilde{h}(\theta,\theta_1)+\partial_1\tilde{h}(\theta,\theta_1)|.
$$
To estimate it we first notice that
$$
|\chi(\theta_1-\theta)\partial_2h(\theta,\theta_1)+\chi(\theta_1-\theta)\partial_1h(\theta,\theta_1)|=|\chi(\theta_1-\theta)||\partial_2h(\theta,\theta_1)+\partial_1h(\theta,\theta_1)|\leq M
$$
using property $6$ in the definition of the class $\mathcal{P}^{\rho_+,\rho_-}$. Moreover, with the change of variable $\theta_1-\xi=\eta-\theta$ we get
$$
|\delta\int^{\theta_1}_{\theta+\tilde{b}-\epsilon} \{\chi(\eta-\theta)-1\}d\eta-\delta\int^{\theta_1-\tilde{b}+\epsilon}_{\theta} \{\chi(\theta_1-\xi)-1\}d\xi|=0.
$$ 
So we just have to estimate the quantity
$$
|\int^{\theta_1-\tilde{b}+\epsilon}_{\theta}\chi^\prime(\theta_1-\xi)\partial_{2}h(\xi,\theta_1)d\xi+ \int^{\theta_1}_{\theta+\tilde{b}-\epsilon}\chi^\prime(\eta-\theta)\partial_{1}h(\theta,\eta)d\eta|
$$
that, after the change of variable $\eta=\xi+\tilde{b}-\epsilon$ in the first integral and having noticed that $|\chi^\prime|$ is bounded, reduces to an estimate of
\begin{equation*}
\begin{split}
&\int^{\theta_1}_{\theta+\tilde{b}-\epsilon}|\partial_{2}h(\eta-\tilde{b}+\epsilon,\theta_1)+\partial_{1}h(\theta,\eta)|d\eta \\
&\leq |\theta_1-\theta-\tilde{b}+\epsilon|\max_{\theta+\tilde{b}-\epsilon\leq\eta\leq\theta_1}|\partial_{2}h(\eta-\tilde{b}+\epsilon,\theta_1)+\partial_{1}h(\theta,\eta)|.
\end{split}
\end{equation*}  
Remembering that we are working in the region $\tilde{b}-\epsilon\leq\theta_1-\theta\leq\tilde{b}$,
\begin{equation}\label{gui}
|\theta_1-\theta-\tilde{b}+\epsilon|\leq\epsilon.
\end{equation}
Now, by the Legendre condition, the function
$$
\Psi(\eta)=\partial_{2}h(\eta-\tilde{b}+\epsilon,\theta_1)+\partial_{1}h(\theta,\eta)
$$
is monotone, so $\max_{\theta+\tilde{b}-\epsilon\leq\eta\leq\theta_1}|\Psi(\eta)|$ is either $|\Psi(\theta_1)|$ or $|\Psi(\theta+\tilde{b}-\epsilon)|$. Suppose we are in the first case, being the other similar. We have
\begin{equation*}
\begin{split}
&|\Psi(\theta_1)|\leq|\partial_{2}h(\theta_1-\tilde{b}+\epsilon,\theta_1)-\partial_{2}h(\theta,\theta_1)|+|\partial_{1}h(\theta,\theta_1)+\partial_{1}h(\theta,\theta_1)|\\
&\leq|\partial_{12}h(c,\theta_1)||\theta_1-\theta-\tilde{b}+\epsilon|+M
\end{split}
\end{equation*}  
for some $c\in [\theta,\theta_1-\tilde{b}+\epsilon]$. Now we can conclude using (\ref{gui}) and (\ref{hbound}).       
\end{proof}

\begin{lemma}\label{estimo}
Let $F(\theta,r)$ be a diffeomorphism of $\mathbb{T}\times\mathbb{R}$. Assume that $F=f_1\circ\dots\circ f_N$ with $f_i\in\mathcal{P}^{\rho_+,\rho_-}$ for $i=1,\dots,N$. Then, for every $\omega\in (N\rho_-,N\rho_+)$ there exists three non negative constant $r_*$, $A$ and $B$ such that
\begin{equation*}
\begin{sistema}
\Theta(\theta,r)  -\theta\geq\omega+\eta \quad\mbox{for } r>r_*  \\
\tilde{\Theta}(\theta,r)  -\theta\geq\omega+\eta \quad\mbox{for } r>r_* \\
\Theta(\theta,r)  -\theta\leq\omega-\eta  \quad\mbox{for } r<-r_*  \\
\tilde{\Theta}(\theta,r)  -\theta\leq\omega-\eta \quad\mbox{for } r<-r_*
\end{sistema}
\end{equation*}
where $\tilde{F}$ has support $[-r_*-A^*,r_*+B^*]$ with $A^*>A$ and $B_*>B$.
\end{lemma}
\begin{proof}
For simplicity of notation, let us prove it for $N=2$. The proof goes by induction. If $N=1$, then $\omega\in (\rho_-,\rho_+)$ and then by property $5'$ in the definition of the class $\mathcal{P}^{\rho_+,\rho_-}$ there exist $r_*>0$ and $\eta>0$ such that
\begin{equation*}
\begin{sistema}
\Theta(\theta,r)  -\theta\geq\omega+\eta \quad\mbox{for } r>r_*  \\
\Theta(\theta,r)  -\theta\leq\omega-\eta  \quad\mbox{for } r<-r_*  
\end{sistema}
\end{equation*}
Every modified $\tilde{F}$ outside $[-r_*,r_*]$ is twist, so
$$
\frac{\partial(\tilde{\Theta}(\theta,r)  -\theta)}{\partial r}>0
$$
and, remembering that $F(\theta,\pm r_*)=\tilde{F}(\theta,\pm r_*)$ for every $\theta$, one can verify that also 
\begin{equation*}
\begin{sistema}
\tilde{\Theta}(\theta,r)  -\theta\geq\omega+\eta \quad\mbox{for } r>r_*  \\
\tilde{\Theta}(\theta,r)  -\theta\leq\omega-\eta  \quad\mbox{for } r<-r_*.  
\end{sistema}
\end{equation*}
Now suppose that $F=f_1\circ f_2$ so that we fix $\omega\in (2\rho_-,2\rho_+)$. From the case $N=1$ there exist $\rho_*$ and $\eta$ such that, for $i=1,2$,   
\begin{equation}\label{indu}
\begin{sistema}
\Theta^{(i)}(\theta,r)  -\theta\geq\frac{\omega+\eta}{2} \quad\mbox{for } r>\rho_*  \\
\tilde{\Theta}^{(i)}(\theta,r)  -\theta\geq\frac{\omega+\eta}{2} \quad\mbox{for } r>\rho_* \\
\Theta^{(i)}(\theta,r)  -\theta\leq\frac{\omega-\eta}{2}  \quad\mbox{for } r<-\rho_*  \\
\tilde{\Theta}^{(i)}(\theta,r)  -\theta\leq\frac{\omega-\eta}{2} \quad\mbox{for } r<-\rho_*.
\end{sistema}
\end{equation}
Moreover, as $f_2$ preserves the end, there exists $r_*>\rho_*$ such that $R^{(2)}(\theta,r)>\rho_*$ for $r>r_*$. 
So, for $r>r_*$
$$
\Theta(\theta,r)-\theta=\Theta^{(1)}(\Theta^{(2)}(\theta,r),R^{(2)}(\theta,r))-\Theta^{(2)}(\theta,r)+\Theta^{(2)}(\theta,r)-\theta\geq\omega+\eta
$$
Analogously we can suppose that
$$
\Theta(\theta,r)-\theta\leq\omega-\eta \quad\mbox{for }r<-r_*. 
$$
Now take the modified $\tilde{f}_i$ with support bigger than $[-r_*-K,r_*+K]$ where $K$ is the constant coming from lemma \ref{estimerre}. Let us estimate the quantity
$$
\tilde{\Theta}(\theta,r)-\theta=\tilde{\Theta}^{(1)}(\tilde{\Theta}^{(2)}(\theta,r),\tilde{R}^{(2)}(\theta,r))-\tilde{\Theta}^{(2)}(\theta,r)+\tilde{\Theta}^{(2)}(\theta,r)-\theta
$$
for $r>r_*$. It comes from (\ref{indu}) that $\tilde{\Theta}^{(2)}(\theta,r)-\theta\geq\frac{\omega+\eta}{2}$. It remains to prove that
$$
\tilde{\Theta}^{(1)}(\tilde{\Theta}^{(2)}(\theta,r),\tilde{R}^{(2)}(\theta,r))-\tilde{\Theta}^{(2)}(\theta,r)\geq\frac{\omega+\eta}{2}.
$$
If $r_*<r\leq r_*+K$ then $\tilde{R}^{(2)}(\theta,r)=R^{(2)}(\theta,r)>\rho_*$ and we get the estimation through (\ref{indu}). If $r>r_*+K$ then, by the definition of $K$, we have $\tilde{R}^{(2)}(\theta,r)>r_*>\rho_*$ and we conclude as before. In an analogous way we have the others estimates.
\end{proof}

\begin{lemma}\label{entra}
Let $F(\theta,r)$ be a diffeomorphism of $\mathbb{T}\times\mathbb{R}$ that possesses an invariant curve $\Gamma$. Assume that $F=f_1\circ\dots\circ f_N$ with $f_i\in\mathcal{P}^{\rho_+,\rho_-}$ for $i=1,\dots,N$. Then, for every $\omega\in (N\rho_-,N\rho_+)$ there exist three non negative constants $r_*$, $A$ and $B$, such that the following holds. Let $(\theta_n,r_n)$ be an orbit of $F$ or of a modified $\tilde{F}$ with support $[-r_*-A^*,r_*+B^*]$ with $A^*>A$ and $B_*>B$. Suppose that for every $\eta>0$ we have
$$
\liminf_{n\to\infty}\frac{\theta_n}{n}<\omega+\eta\quad\mbox{and}\quad \limsup_{n\to\infty}\frac{\theta_n}{n}>\omega-\eta
$$
then there exists ${\bar{n}}\in\mathbb{Z}$ such that
$$
(\theta_{\bar{n}},r_{\bar{n}})\in\mathbb{T}\times(-r_*,r_*).
$$   
\end{lemma}
\begin{proof}
Let $r_*$, $A$ and $B$ the constant coming from lemma \ref{estimo}. We can suppose that $r_*$ is large enough to have $\Gamma\subset\mathbb{T}\times(-r_*,r_*)$. The invariant curve divides the cylinder in two components, the upper and the lower and both are $F$-invariant (resp. $\tilde{F}$-invariant). Notice that to prove it we must use the fact that $F$ (resp. $\tilde{F}$) preserves the ends. It means that there cannot exist orbits that jump from the top to the bottom of the cylinder. So, if an orbit $(\theta_n,r_n)$ of $F$ or $\tilde{F}$ is such that $r_n>r_*$ for every $n$ or $r_n<-r_*$ for every $n$ then    
$$
\liminf_{n\to\infty}\frac{\theta_n}{n}\geq\omega+\eta\quad\mbox{or}\quad \limsup_{n\to\infty}\frac{\theta_n}{n}\leq\omega-\eta
$$
respectively, in contradiction with the hypothesis.
\end{proof}

%
%
%

Now we are ready for the

\begin{proof}[Proof of theorem \ref{comp}]
Fix $\omega\in (N\rho_-,N\rho_+)$, consider the constants $r_*$, $A$ and $B$ coming from lemma \ref{entra}.  By hypothesis, we can find two invariant curves $\Gamma_+$ and $\Gamma_-$ contained, respectively in $r>r_*$ or $r<r_*$. Let $\Sigma$ be the compact region defined by such curves. Let $F^{(j)}=f_1\circ\dots\circ f_j$ for $j=1,\dots,N$. The sets $F^{(j)}(\Sigma)$ are compacts and so one can find a region $\tilde{\Sigma}$, defined by two invariant curves such that  
\begin{equation*}
\Sigma\cup F^{(1)}(\Sigma)\cup F^{(2)}(\Sigma)\cup\dots\cup F^{(N)}(\Sigma)\subset int\tilde{\Sigma}.
\end{equation*}
Analogously, we can find and a region $\Sigma_1=\mathbb{T}\times [-r_*-A^*,r_*+B^*]$ with $A_*>A$ and $B_*>B$ such that
\begin{equation*}
\tilde{\Sigma}\cup F^{(1)}(\tilde{\Sigma})\cup F^{(2)}(\tilde{\Sigma})\cup\dots\cup F^{(N)}(\tilde{\Sigma})\subset int\Sigma_1.
\end{equation*}
Now modify every $f_i$ outside the strip $\Sigma_1$ applying lemma \ref{modif} and find the corresponding $\tilde{f}_i$. So we get $\tilde{F}=\tilde{f}_1\circ\dots\circ \tilde{f}_n$. The diffeomorphisms $\tilde{F}$ satisfies the hypothesis of theorem \ref{matteo} so we get an orbit $(\bar{\theta}_n,\bar{r}_n)$ of $\tilde{F}$ with rotation number $\omega$. By lemma \ref{entra} there exists $\bar{n}$ such that $(\bar{\theta}_{\bar{n}},\bar{r}_{\bar{n}})\in\Sigma$. Notice that $\Gamma_+$ and $\Gamma_-$ are also invariant curves for $\tilde{F}$ and so by the invariance on $\Sigma$ we have that $(\bar{\theta}_n,\bar{r}_n)\in\tilde{\Sigma}$ for every $n$. But in $\tilde{\Sigma}$ we have $F=\tilde{F}$ so that $(\bar{\theta}_n,\bar{r}_n)$ is also an orbit of $F$. Remembering corollary \ref{relmat} we get the thesis.
\end{proof}

Finally, we are ready for

\begin{proof}[Proof of theorem \ref{teo3}]
From proposition \ref{pll} we can apply theorem \ref{comp} to the Poincaré map $\Pi$ of system (\ref{rpend4}) and find for every $\omega\in (-T,T)$ two functions $\phi$ and $\eta$ such that
\begin{equation}\label{perm}
\phi(\xi+1)=\phi(\xi)+1,\quad \eta(\xi+1)=\eta(\xi)
\end{equation}
\begin{equation}\label{poin}
\Pi(\phi(\xi),\eta(\xi))=(\phi(\xi+\omega),\eta(\xi+\omega)).
\end{equation}
Let $X_\xi(t)=(Q_\xi(t),P_\xi(t))$ be the solution of (\ref{rpend4}) with initial condition $(\phi(\xi),\eta(\xi))$. Notice that from (\ref{perm}) and uniqueness we have that
$$
X_{\xi+1}(t)=X_\xi(t)+(1,0)
$$
and from (\ref{poin}) and the definition of $\Pi$,
$$
X_\xi(t+T)=X_{\xi+\omega}(t).
$$
so that (\ref{piu}) is verified.
Finally, consider the limit 
$$
\lim_{t\to\infty}\frac{Q_\xi(t)}{t}.
$$
We have that, for $nT\leq t\leq (n+1)T$
$$
\frac{Q_\xi(t)}{t}=\frac{Q_\xi(t)-Q_\xi(nT)}{t}+\frac{Q_\xi(nT)}{nT}\frac{nT}{t}
$$ 
where, being the vector field in (\ref{rpend4}) bounded, the quantity $Q_\xi(t)-Q_\xi(nT)$ is bounded.
So we can compute
$$
\lim_{t\to\infty}\frac{Q_\xi(t)}{t}=\lim_{n\to\infty}\frac{Q_\xi(nT)}{nT}=\lim_{n\to\infty}\frac{Q_{\xi+n\omega}(0)}{nT}= \lim_{n\to\infty}[\frac{Q_{\xi+\{n\omega\}}(0)}{nT}+\frac{[n\omega]}{nT}]=\frac{\omega}{T}
$$ 
where $[x]$ denotes the integer part of $x$ and $\{x\}=x-[x]$.
\end{proof}

\section*{Acknowledgements}
I wish to thank Professor Rafael Ortega for his constant supervision and support.

\bibliographystyle{plain}
\bibliography{biblio4}

\end{document}